\documentclass[graybox]{svmult}

\usepackage{mathptmx}       %

\usepackage{helvet}         %
\usepackage{courier}        %

\usepackage{graphicx}              %

\usepackage{siunitx}
\usepackage{amsbsy}
\usepackage{amsmath,bm}
\usepackage{amssymb}
\usepackage{amsfonts}
\usepackage{mathtools}
\usepackage{tikz}

\usepackage{multirow}

\usepackage[bottom]{footmisc}%
\usepackage{cite}

\usepackage{url}

\DeclareMathOperator{\curl}{curl}

\DeclareMathOperator{\N}{\mathcal{N}}
\DeclareMathOperator{\B}{\mathcal{B}}
\newcommand{\dtt}{\partial_{tt}}
\newcommand{\dt}{\partial_{t}}

\renewcommand{\E}{\mathbf{E}}
\newcommand{\RR}{\mathbb{R}}
\newcommand{\Th}{\mathcal{T}_h}
\newcommand{\Vh}{\mathcal{V}_h}

\begin{document}

\title*{A Yee-like finite element scheme for Maxwell's equations on hybrid grids}
\author{Herbert Egger \inst{1}
\and
Bogdan Radu \inst{2}
}
\institute{Herbert Egger \at Institute for Numerical Mathematics, Johannes-Kepler University Linz, Austria \\
\email{herbert.egger@jku.at}
\and Bogdan Radu \at Johann Radon Institute for Computational and Applied Mathematics, Linz, Austria \\
\email{bogdan.radu@ricam.oeaw.ac.at}}

\maketitle

\abstract{
A novel finite element method for the approximation of Maxwell's equations over hybrid two-dimensional grids is studied. The choice of appropriate basis functions and numerical quadrature leads to diagonal mass matrices which allow for efficient time integration by explicit methods. %
On purely rectangular grids, the proposed schemes coincide with well-established FIT and FDTD methods. Additional internal degrees of freedom introduced on triangles allow for mass-lumping without the usual constraints on the shape of these elements. 
A full error analysis of the method is developed and numerical tests are presented for illustration.
}

\abstract*{
A finite element method for the approximation of Maxwell's equations over hybrid two-dimensional grids is studied. The choice of appropriate basis functions and numerical quadrature leads to block-diagonal mass matrices which allow for efficient time integration by explicit methods. %
On rectangular grids, the proposed schemes coincide with well-established FIT and FDTD methods. Additional internal degrees of freedom allow for mass-lumping without severe constraints on the shape of the triangles. %
A full error analysis of the method is given and numerical tests are presented for illustration.
}

\section{Introduction} \label{sec:intro}
The propagation of electromagnetic waves through a non-dispersive linear medium can be described by the time-dependent Maxwell's equations
\begin{alignat}{2}
  \varepsilon\dt E + \curl H &= -j,   \label{sys1}\\ 
          \mu\dt H + \curl E &= 0,   \label{sys2}
\end{alignat}
together with appropriate initial and boundary conditions. 
Here $E$, $H$ denote the electric and magnetic field intensities, $\varepsilon$, $\mu$ the corresponding material parameters, and $j$ describes the density of source and eddy currents. 
An efficient discretization of \eqref{sys1}--\eqref{sys2} can be achieved by the finite difference time domain (FDTD) method or the finite integration technique (FIT), see e.g. \cite{Yee66,Weiland03},
and for isotropic materials and orthogonal grids, second-order convergence can be obtained in space and time.
In order to handle complex geometries, several attempts have been made to generalize these methods to non-orthogonal and unstructured grids; see e.g. \cite{BossavitKettunen99,RylanderBondeson00,vanRienen2004} and also \cite{CodecasaPoliti08,CodecasaKapidaniSpecognaTrevisan18} for more recent results. A rigorous error analysis of a Yee-like scheme on triangles and tetrahedra was given in \cite{EggerRadu20c}, and first-order convergence in space on general unstructured grids was demonstrated theoretically and numerically. 

\smallskip 

\textbf{Scope.}
In this paper, we propose a novel Yee-like discretization scheme for hybrid grids in two space dimensions, consisting of triangles and rectangles.
The method is based on a finite element approximation with mass-lumping through numerical quadrature, which allows for a rigorous error analysis; see \cite{Cohen02,Monk92a} for background. 
On rectangular grid cells, the resulting discretization coincides with that of the FIT or FDTD method. Following \cite{ElmkiesJoly97b}, additional internal degrees of freedom are introduced on triangular grid cells, which allows us to prove discrete stability  without severe restrictions on the mesh.    
The lowest order approximation on two-dimensional hybrid grids is studied in detail.
The main ideas behind the construction of the method and its analysis however carry over to three dimensions and higher-order approximations; see \cite{EggerRadu21a,ElmkiesJoly97b,Radu22} and the discussion at the end of the paper.

\section{Description of the problem}
Let us start with completely specifying the model problem to be considered in the rest of the paper. 
We choose $\varepsilon=\mu=1$ and abbreviate $f=-\dt j$. 
Moreover, we consider the second-order form of Maxwell's equations, i.e.,
\begin{alignat}{2}
  \dtt  E +\curl(\curl E) &= f, \qquad &&\text{in }\Omega,  \label{s1}\\ 
  n\times \curl(E) &= 0,\qquad &&\text{on }\partial\Omega, \label{s2}
\end{alignat}
with simple boundary conditions.
The computational domain  $\Omega\subseteq\RR^2$ is assumed to be a bounded Lipschitz polygon and $\curl E = \partial_x E_2 - \partial_y E_1 $ denotes the curl of a vector field $E=(E_1,E_2)$ in two space dimensions.
The above differential equations are considered on a finite time interval $[0,T]$, and complemented by suitable initial conditions $E(0)=E_0$ and $\dt E(0) = E_1$.
The existence of a unique solution can then be established by sem-group theory or Galerkin approximation.
Solutions of \eqref{s1}--\eqref{s2} can further be characterized equivalently by the variational identities
\begin{align}
  (\dtt E(t),v)+(\curl E(t),\curl v) = (f(t),v), \label{sys5}
\end{align}
for all $v\in H(\curl,\Omega)= \{E\in L^2(\Omega)^2\,:\, \curl E\in L^2(\Omega)\}$ and a.a. $t \in [0,T]$. 
For abbreviation, we write $(a,b)=\int_\Omega a \cdot b \, dx$ for the scalar product on $L^2(\Omega)$ and $L^2(\Omega)^2$.

\section{A finite element method with mass-lumping}
Let $\Th$ = $\{K\}$ be a quasi-uniform shape-regular partition mesh of $\Omega$ into triangular and/or rectangular elements $K$. Different elements are allowed to meet only at edges or vertices. 
By assumption, all edges of the mesh are of similar length and we call the size $h$ of the longest edge in the mesh the global mesh size. 
\smallskip 

\noindent 
\textbf{Finite element spaces.}
For the approximation of the field $E$ on individual elements, we consider local polynomial spaces defined by
\begin{alignat}{2}
V(K) = \left\{\begin{array}{ll}
\N_0(K), \qquad & \text{ if $K$ is a square,} \\[.3em]
\N_0^+(K) = \N_0(K) + \B(K), \qquad & \text{ if $K$ is a triangle}.	
\end{array}\right. \label{sys8}
\end{alignat}
Here $\N_0(K)$ is the lowest order Nedelec space for triangles or rectangles \cite{BoffiBrezziFortin13,Nedelec80}, and $\B(K)$ is a space of three quadratic functions with vanishing tangential components. 
The corresponding degrees of freedom are depicted in Figure~\ref{fig:elements1}, and details on the basis functions are presented in Section~\ref{sec:basis}. 
\begin{figure}[ht!]
  \centering
  \begin{tikzpicture}[scale=0.6]
    \draw (0,0) -- (3,0);
    \draw (0,0) -- (0,3);
    \draw (3,0) -- (3,3);
    \draw (0,3) -- (3,3);
    
    \draw[thick,->] (2,-0.25) -- (1,-0.25);
    \draw[thick,->] (-0.25,1) -- (-0.25,2);
    
    \draw[thick,->] (1,3.25) -- (2,3.25);
    \draw[thick,->] (3.25,2) -- (3.25,1);
    
    \draw[fill,blue] (3,1.5) circle (0.12cm);
    \draw[fill,blue] (1.5,0) circle (0.12cm);
    \draw[fill,blue] (0,1.5) circle (0.12cm);
    \draw[fill,blue] (1.5,3) circle (0.12cm);    
  \end{tikzpicture}
  \qquad\qquad\qquad
  \begin{tikzpicture}[scale=0.6]
    \draw (0,0) -- (3,0);
    \draw (0,0) -- (0,3);
    \draw (3,0) -- (0,3);
    \draw[thick,->,red] (1.5,1.5) -- (1,1);
    \draw[thick,->] (1.2,2.1) -- (2,1.3);
    
    \draw[thick,->,red] (0,1.5) -- (0.70,1.5);
    \draw[thick,->] (2,-0.25) -- (1,-0.25);
    
    \draw[thick,->,red] (1.5,0) -- (1.5,0.70);
    \draw[thick,->] (-0.25,1) -- (-0.25,2);
    
    \draw[fill,blue] (1.5,1.5) circle (0.12cm);
    \draw[fill,blue] (1.5,0) circle (0.12cm);
    \draw[fill,blue] (0,1.5) circle (0.12cm);
  \end{tikzpicture}
  \caption{Degrees of freedom for the space $\N_0(K)$ on the rectangle (left) and the space $\N_0^+(K)$ on the triangle (right). The three internal degrees of freedom for the bubble functions are displayed in red and the corresponding quadrature points are depicted as blue dots. \\[-2em]} \label{fig:elements1}
\end{figure}
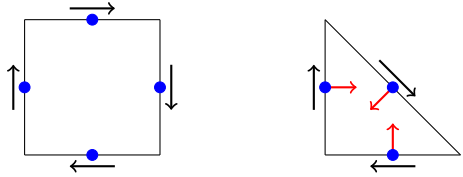%
Let us note that the finite element space $\N_0^+(K)$ was originally introduced in \cite{ElmkiesJoly97b}.
The global finite element space induced by the local spaces $V(K)$ is defined by
\begin{align*}
\Vh = \{v_h\in H(\curl;\Omega)\,:\,\, v_h|_K\in V(K) \ \forall K \in \Th\}. 
\end{align*}

\noindent 
\textbf{Quadrature.}
We use an approximation $(u,v)_{h}:=\sum_{K}(u,v)_{h,K}$ for the $L^2$-scalar product, with contributions obtained by numerical integration. On the triangle, we set
\begin{align}
(u,v)_{h,K} = |K| \sum_{i=1}^{3}\;\tfrac{1}{3}\,u(m_{K,i})\cdot v(m_{K,i}),
\label{sys9}
\end{align}
where $m_{K,i}$ is the midpoint of the edge $e_i$ opposite to vertex $i$; see Figure~\ref{fig:elements1}.
For the rectangle, we proceed differently: Here we decompose $(u,v) = (u_1, v_1) + (u_2,v_2)$ into two contributions for the orthogonal directions, and then use different quadrature rules for the two contributions, i.e.
\begin{align}
(u,v)_{h,K} = |K| \left( \sum_{i=1}^2 \tfrac{1}{2} u_1(m_{K,h,i}) v_1(m_{K,h,i}) + \sum_{j=1}^2 \tfrac{1}{2} u_2(m_{K,v,j}) v_2(m_{K,v,j})   \right).
\label{sys9}
\end{align}
Here $m_{K,h,i}$ and $m_{K,h,j}$ are the midpoints of the horizontal and vertical edges, respectively; see again Figure~\ref{fig:elements1}. 
For the semi-discretization of our model problem in space, we then consider the following inexact Galerkin approximation.
\begin{problem}\label{p:2}
Let $E_{h,0}$, $E_{h,1} \in \Vh$ be given. 
Find $E_h:[0,T]\rightarrow \Vh$ such that
\begin{align}
(\partial_{tt} E_h(t),v_h)_h + (\curl E_h(t),\curl v_h) = (f(t), v_h)  \label{sys6}
\end{align}
for all $v_h\in \Vh$ and all $t\in[0,T]$, and such that $E_h(0)=E_{h,0}$ and $\partial_{t} E_h(0)=E_{h,1}$.
\end{problem}
As we will indicate below, the implementation of this method leads to a diagonal mass matrix, which allows using explicit methods for efficient time integration.

\vspace*{-1em}

\section{Preliminary results}

By elementary computations, one can verify the following assertions, which ensure the well-posedness of Problem~\ref{p:2} and yield a starting point for our error analysis.
\begin{lemma}\label{lem:equiv}
The quadrature rule~\eqref{sys9} is exact for polynomials of degree $k\le 2$ on triangles and for polynomials of degree $k\le 1$ on squares. 
Moreover, the inexact scalar product $(\cdot,\cdot)_h$ induces a norm $\|\cdot\|_h$ on $\Vh$, which is equivalent to the $L^2$-norm on $\Vh$, and consequently Problem~\ref{p:2} has a unique solution. 
\end{lemma}

As a second ingredient, let us recall some results about polynomial interpolation. 
We denote by $\Pi_h : H^1(\Th)^2 \to \Vh$ the projection defined element-wise by
\begin{align}\label{sys8}
(\Pi_h E)|_K \coloneqq \Pi_K E|_K
\end{align}
where $\Pi_K : H^1(K) \to \N_0(K)$ is the standard interpolation operator for the lowest order Nedelec space $\N(K)$ on both triangles and squares; see \cite{BoffiBrezziFortin13,Nedelec80} for details.
We further denote by $\pi_h^0 : L^2(\Omega) \to P_0(\Th)$ the $L^2$-orthogonal projection onto piecewise constants; the same symbol is used for the projection of vector-valued functions.
\begin{lemma}\label{lem:interp}
Let $K\in \Th$ and $\Pi_h$ defined as in \eqref{sys8}.
Then
\begin{align}
  \|E-\Pi_h E\|_{L^2(K)}&\leq Ch\|E\|_{H^1(K)}, \label{eq:interperr1}\\
  \|\curl(E-\Pi_h E)\|_{L^2(K)}&\leq Ch\|\curl E\|_{H^1(K)}, \label{eq:interperr2}\\
    \|E-\pi_h^0E\|_{L^2(K)}&\le Ch\|E\|_{H^1(K)}, \label{eq:l2err}
\end{align}
whenever $E$ is regular enough, with a constant $C$ independent of $h$. 
\end{lemma}

\noindent 
Having introduced all the required tools, we can now state and prove our main result.
\begin{theorem}\label{thm:err}
Let $E$ and $E_h$ denote the solutions of \eqref{sys5} and \eqref{sys6} with initial values set by $E_h(0)= \Pi_h E(0)$ and $\partial_t E_h(0) = \Pi_h \partial_t E(0)$. 
Then
\begin{align*}
\|\dt(E - E_h)\|_{L^\infty(0,T;L^2(\Omega))} + \|\curl(E - E_h)\|_{L^\infty(0,T;L^2(\Omega))} \leq C(E,T) \, h^2
\end{align*}
with constant $C$ depending on the norm of $E$ but independent of the mesh size $h$. 
\end{theorem}

\section{Proof of Theorem~\ref{thm:err}}

Apart from some technical details, the following analysis follows by standard arguments. For completeness and convenience of the reader, we present all the details.

\smallskip 
\noindent
\textbf{Step~1. Error splitting and estimate for the projection error.}
In the usual manner, we begin by splitting the overall discretization error via
\begin{align}
    E-E_h = (E-\Pi_h E) + (\Pi_h E - E_h) =: -\eta + \psi_h,
\end{align}
into a projection error and a discrete error component. 
By the estimates of Lemma~\ref{lem:interp}, we immediately obtain
\begin{align*}
\|\dt\eta\|_{L^\infty(0,T;L^2(\Omega))} &+ \|\curl\eta\|_{L^\infty(0,T;L^2(\Omega))} \\
&\leq Ch\left(\|\dt E\|_{L^\infty(0,T;H^1(\Th))} + \|\curl E\|_{L^\infty(0,T;H^1(\Th))}\right),
\end{align*}
which already covers the first error component.

\smallskip 

\noindent 
\textbf{Step~2. Discrete error equation.}
By subtracting \eqref{sys9} from \eqref{sys5} with $v=v_h$, 
we can see that the discrete error $\psi_h$ satisfies the identity
\begin{align*}
  (\dtt\psi_h(t),v_h)_h&+(\curl\psi_h(t),\curl v_h) = \\ &(\dtt\eta(t),v_h)+(\curl\eta(t),\curl v_h) +
  \sigma_h(\Pi_h\dtt u(t),v_h)
\end{align*}
for all $v_h \in \Vh$ and $0 \le t \le T$, with quadrature error
\begin{align} \label{sys10}
\sigma_h(E,v) = (E,v)_{h} - (E,v).
\end{align}
We can further split $\sigma_h(E,\phi) = \sum\nolimits_{K\in \Th}\sigma_K(E,\phi)$ into element contributions defined by $\sigma_K(E,\phi) = (E,\phi)_{h,K}-(E,\phi)_K$.
Moreover, 
$\psi_h(0)=\partial_{t} \psi_h(0)=0$, due to the choice of initial conditions for the discrete problem.

\smallskip 

\noindent 
\textbf{Step~3. Estimates for the quadrature error.}
To further proceed in our analysis, we now quantify the local quadrature error in more detail.
\begin{lemma}\label{lem:quaderr}
Let $E \in L^2(\Omega)^2$ with $E|_K \in H^1(K)^2$ for all $K \in \Th$. 
Then 
\begin{alignat*}{2}
|\sigma_K(\Pi_h E, \phi_h)|\leq C h\|E\|_{H^1(K)}\|\phi_h\|_{L^2(K)}
\end{alignat*}
for all $\phi_h\in \Vh$ and all $K\in \Th$ with constant $C$ independent of the element $K$.
\end{lemma}
\begin{proof}
Using Lemma~\ref{lem:equiv}, we deduce that $(u_h^0,v_h)_K=(u_h^0,v_h)_{h,K}$ for all $u_h^0 \in P_0(K)^2$ and $v_h \in V(K)$. We can then estimate the quadrature error by
\begin{align*}
|\sigma_K(\Pi_h u, v_h)|
 &=|\sigma_K(\Pi_h u-\pi_h^0u,v_h)|\le c\|\Pi_h u-\pi_h^0u\|_{L^2(K)}\|v_h\|_{L^2(K)}\\
&\le c'h\|u\|_{H^1(K)}\|v_h\|_{L^2(K)},
\end{align*}
where we used the Cauchy-Schwarz inequality and the norm equivalence of Lemma~\ref{lem:equiv} and the approximation properties of the projections from Lemma~\ref{lem:interp}.
\end{proof}

\smallskip 
\noindent
\textbf{Step~4. Estimate for the discrete error.}
Taking $v_h=\partial_t \psi_h(t)$ as test function in the discrete error equation and integrating from $0$ to $t$ leads to
\begin{align} \label{eq:inequality} 
&\frac{1}{2}\left(\|\dt\psi_h(t)\|_h^2+\|\curl\psi_h(t)\|_{L^2(\Omega)}^2\right)\\
&\quad = \smallint_0^t (\dtt \psi_h(s), \dt \psi_h(s))_h + (\curl \psi_h(s), \curl \dt \psi_h(s)) \, ds \notag\\
&\quad =\smallint_0^t(\dtt\eta(s),\dt\psi_h(s)) + (\curl\eta(s),\curl\dt\psi_h(s)) + 
\sigma_h(\Pi_h\dtt u(s),\dt\psi_h(s)) \, ds \notag.
\end{align}
The three terms can now be estimated separately. 
Using Cauchy-Schwarz and Young inequalities, the first term may be bounded by
\begin{align*}
(i) \leq ch^2\|\dtt E\|_{L^1(0,t,H^1(\Th))}^2 + \tfrac{1}{4}\|\partial_{t}\psi_h\|^2_{L^\infty(0,t,L^2(\Omega))}.
\end{align*}
For the second term, we utilize that 
\begin{align*}
(ii) &= \smallint_0^t (\curl (E - \Pi_h E), \curl \dt \psi_h)) \, ds \\
&= (\curl (E - \Pi_h E)(t), \curl \psi_h(t)) - \smallint_0^t (\curl (\dt E - \Pi_h \dt E), \curl \psi_h) \, ds \\
& \le C h^2 \big(\|\curl E\|_{L^\infty(0,t;H^1(\Th))}^2 + \|\curl \dt E\|_{L^1(0,t;L^2(\Omega))}^2 \big) + \tfrac{1}{4} \|\psi_h\|_{L^\infty(0,t;L^2(\Omega))}^2.
\end{align*}%
The third term can finally be estimated using Lemma~\ref{lem:quaderr} according to
\begin{align*}
(iii)&=\smallint_0^t ch\|\dtt E(s)\|_{H^1(\Omega)}\|\dt\psi_h(s)\|_{L^2(\Omega)}\,ds \\
&\leq ch^2\|\dtt E\|_{L^1(0,t,H^1(\Omega))}^2 + \tfrac{1}{4}\|\partial_{t}\psi_h\|^2_{L^\infty(0,t,L^2(\Omega))}
\end{align*}
Using these estimates in the inequality \eqref{eq:inequality}, absorbing all the terms with the test function into the left side, and taking the supremum over $t\in[0,T]$, after applying the norm equivalence of Lemma~\ref{lem:equiv} to some terms, then leads to the estimate 
\begin{align*}
&\|\dt\psi_h\|_{L^\infty(0,T;L^2(\Omega))}^2 + \|\curl\psi_h\|_{L^\infty(0,T;L^2(\Omega))}^2 \\
& \qquad \leq C h^2\big(\|\dtt E\|_{L^1(0,T;H^1(\Th))}^2 + \|\dt E\|_{L^1(0,T;H^1(\Th))}^2 + \|\curl E\|_{L^\infty(0,T;H^1(\Th))}^2\big)
\end{align*}
for the discrete error component; one may also take the square root in all terms.

\smallskip 
\noindent 
\textbf{Step~5.}
The proof of the theorem is completed by applying the triangle inequality to the error splitting in Step~1 and adding up the estimates for the projection error $\eta$ and the discrete error component $\psi_h$.

\section{Implementation} \label{sec:basis}

For completeness of the presentation, let us briefly discuss the choice of basis functions for the local finite element spaces $\N_0(K)$ and $\N_0^+(K)$ which, together with the numerical quadrature leads to diagonal mass matrices. 

\smallskip 
\noindent 
\textbf{Rectangle.}
On quadrilateral elements $K$, we choose the standard basis for the lowest order Nedelec space $\N_0(K)=\text{span}\{\Phi_{h,i}, \Phi_{v,i}: i=1,\ldots,2\}$; see \cite{BoffiBrezziFortin13,Nedelec80}. These functions have the following properties: The function
$\Phi_{h,i}$ associated to a horizontal edge $e_{h,i}$ vanishes identically on the opposite horizontal edge, and  $\Phi_{v,j}$ associated to for the vertical edge $e_{v,j}$ vanishes on the opposite vertical edge. 
Hence the local mass matrix produced by the quadrature rule $(u,v)_{K,h}$ for every rectangle is diagonal.

\smallskip 
\noindent 
\textbf{Triangle.}
Let $\{\lambda_i\}$ be the barycentric coordinates of the element $K$.
For every edge $e_k=e_{ij}$ pointing from vertex $i$ to  $j$, and thus opposite to $k$, we define the two basis functions
\begin{align*}
\Phi_{ij}^{B} &= \lambda_i \lambda_j \nabla \lambda_k
\qquad \text{and} \qquad \\
\Phi_{ij} &= \lambda_i \nabla \lambda_j - \lambda_j \nabla \lambda_i + \alpha_{ij} \Phi_{ij}^{B} + \beta_{ij} \Phi_{jk}^{B} + \gamma_{ij} \Phi_{ki}^{B}.
\end{align*}
Then $\N_0^+(K)=\operatorname{span}\{\Phi_{12},\Phi_{23},\Phi_{31},\Phi_{12}^{B},\Phi_{23}^{B},\Phi_{31}^{B}\}$.
The bubble functions $\Phi_{ij}^{B}$ have vanishing tangential components on the edge are $e_k$, and they vanish identically on the two remaining edges $e_i$, $e_j$. 
The functions $\Phi_{ij}$ are modified Nedelec basis functions. They have vanishing tangential components on the two edges $e_i$, $e_j$, and by appropriate choice of the parameters $\alpha_{ij}$, $\beta_{ij}$, $\gamma_{ij}$, their normal components on all edge midpoints $m_{K,i}$ can be made zero. 
As a consequence, the local mass matrix produced by the scalar product $(u,v)_{K,h}$ for the triangle becomes diagonal.

\medskip 
\noindent
\textbf{Summary.}
The global mass matrix is obtained by assembling the local mass matrices, which are diagonal, and hence has inherits this property.

\section{Numerical illustration}
We consider the computational domain $\Omega = \Omega_1\cup\Omega_2$ where $\Omega_1 = (0,2)\times (-1,1)$ and  $\Omega_2 = \big((2,4)\times (-1,1))\setminus B_{0.3}(3,0)$, where $B_{r}(x,y)$ denotes the ball with radius $r$ around midpoint $(x,y)$.
The two subdomains are meshed by rectangles and triangles, respectively. 
For our test problem, we consider the wave equation \eqref{s1}. The boundary $\partial \Omega$ is split into several parts and as boundary conditions, we impose
\begin{alignat*}{2}
n \times E &= \sin(10\cdot t) \cdot e^{-10y^2}, \qquad && \text{on } \partial\Omega_{\text{left}}, \\ 
n \times E &= 0, \qquad && \text{on } \partial\Omega_{\text{ball}},\\ n \times \curl E &= 0, \qquad && \text{else.}    
\end{alignat*}
The initial conditions are chosen as $E(0)=\dt E(0)=0$. 
This corresponds to a pulse entering at the left boundary, propagating through the domain, and getting reflected at the 
walls of the box and the circular inclusion.
Some snapshots of the solution are depicted in Figure~\ref{fig:result}.
\begin{figure}
    \centering
    \includegraphics[scale=0.24]{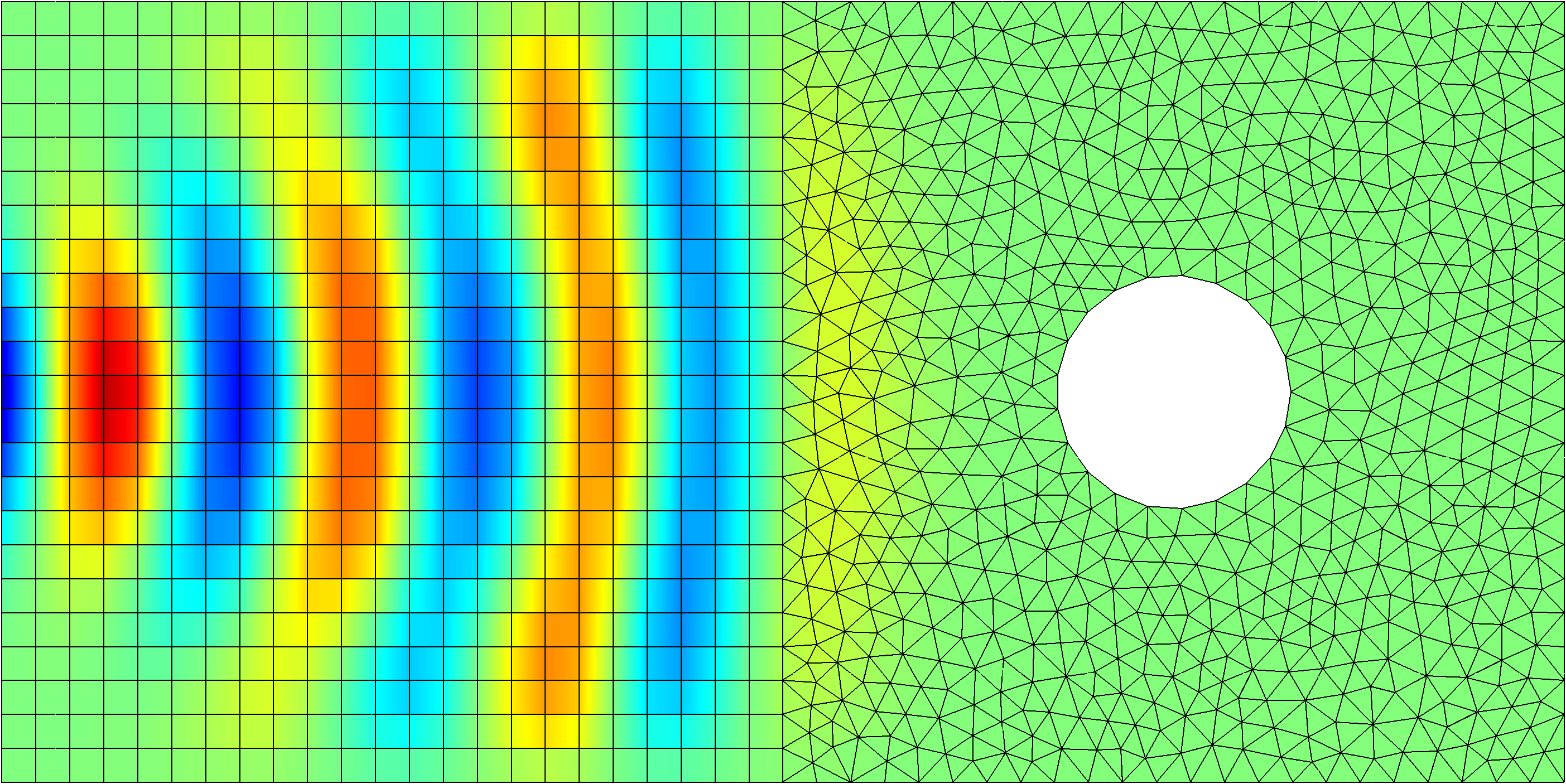}
    \;
    \includegraphics[scale=0.24]{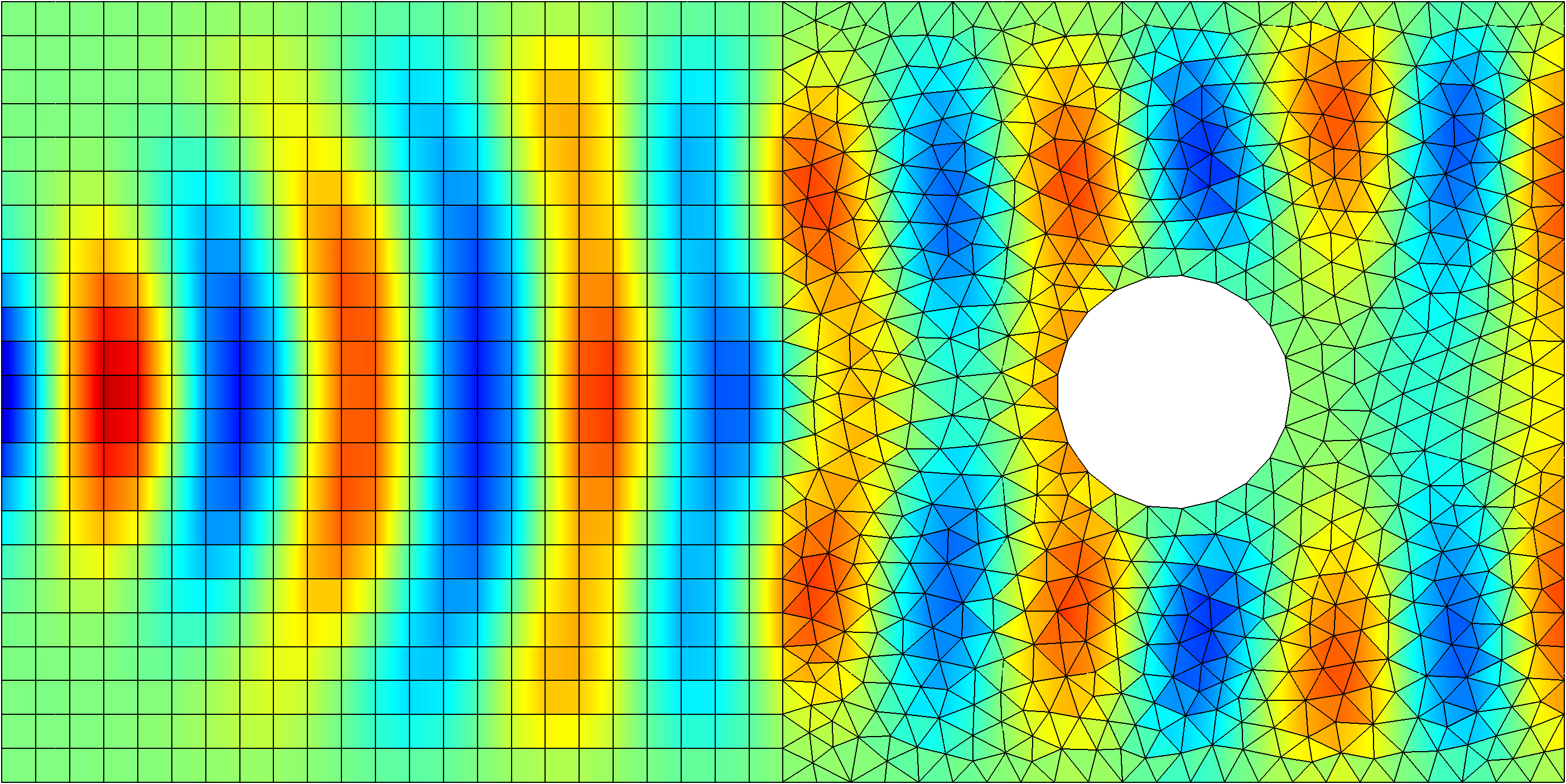}
    \caption{The first $E_1$ component of the solution $E=(E_1,E_2)$ at time steps $t=2.3$ and $t=5$ showing the scattering at the sphere.}
    \label{fig:result}
\end{figure}
Let us remark that no reflections are observed at the interface between the two meshes. 
In our numerical tests, we observe linear convergence $O(h)$ of the error. This coincides with the theoretical predictions of Theorem~\ref{thm:err}, and also demonstrates that the error estimates are sharp. 
Note that second order convergence is in general lost for Yee-like approximations on unstructured grids; also see \cite{Radu22}.

\end{document}